\documentclass[reqno,centertags,11pt,draft]{amsart}
\usepackage{amsmath,amsthm,amsfonts,amssymb,enumerate}
\usepackage[bookmarksopen=true,final]{hyperref}
\usepackage{multirow,yhmath}
\usepackage{xcolor}
\usepackage{bm}
\usepackage[nobysame]{amsrefs}

\newcommand{\bbC}{{\mathbb{C}}}

\newcommand{\bbN}{{\mathbb{N}}}

\newcommand{\bbR}{{\mathbb{R}}}

\newcommand{\calM}{{\mathcal M}}
\newcommand{\calN}{{\mathcal N}}

\newcommand{\calZ}{{\mathcal Z}}

\newcommand{\la}{\left\langle}
\newcommand{\ra}{\right\rangle}

\newcommand{\supp}{\text{\rm{supp}}}

\newcommand{\beq}{\begin{equation}}
\newcommand{\eeq}{\end{equation}}
\newcommand{\ba}{\begin{align*}}
\newcommand{\ea}{\end{align*}}

\newcommand{\abs}[1]{\lvert#1\rvert}

\DeclareFontFamily{U}{mathx}{}
\DeclareFontShape{U}{mathx}{m}{n}{<-> mathx10}{}
\DeclareSymbolFont{mathx}{U}{mathx}{m}{n}
\DeclareMathAccent{\widehat}{0}{mathx}{"70}
\DeclareMathAccent{\widecheck}{0}{mathx}{"71}

\allowdisplaybreaks
\numberwithin{equation}{section}

\usepackage[normalem]{ulem}

\newtheorem{thm}{Theorem}[section]
\newtheorem{lem}[thm]{Lemma}
\newtheorem{lemma}[thm]{Lemma}
\newtheorem{cor}[thm]{Corollary}
\newtheorem{prop}[thm]{Proposition}
\newtheorem{rem}[thm]{Remark}
\newtheorem{defn}[thm]{Definition}

\begin{document}
\title[Zeros of MOPRL: location and interlacing]{Zeros of multiple orthogonal polynomials: location and interlacing}
\author{Rostyslav Kozhan and Marcus Vaktnäs}
\date{\today}
\begin{abstract}
	We prove a criterion on the possible locations of zeros of type I and type II multiple orthogonal polynomials in terms of normality of degree $1$ Christoffel transforms. 
    We provide another criterion in terms of degree $2$ Christoffel transforms for establishing zero interlacing of the neighbouring  multiple orthogonal polynomials of type I and type II.



    We apply these criteria to establish zero location and interlacing of type I multiple orthogonal polynomials for Nikishin systems. Additionally, we  recover the known results on zero location and interlacing for type I multiple orthogonal polynomials for Angelesco systems, as well as for type II multiple orthogonal polynomials for Angelesco and AT systems. 

    Finally, we demonstrate that normality of the higher order Christoffel transforms naturally related to the zeros of the Wronskians of consecutive orthogonal polynomials.
\end{abstract}
\maketitle
\section{Introduction}

Given a set $\Gamma$ on the real line $\bbR$, 
let  $\calM(\Gamma)$ be the set of all finite Borel measures $\mu$ of constant sign, 
with all the moments finite, and with an infinite support  $\supp\,\mu\subseteq \Gamma$. 
For any $\mu \in\mathcal{M}(\bbR)$, its orthogonal polynomial of degree $n\in \bbN:=\{0,1,2,\ldots\}$ is defined to be the monic polynomial $P_n$ of exact degree $n$ satisfying
\begin{equation}\label{eq:OPRL}
		\int P_{n}(x)x^k d\mu(x) = 0,\qquad k = 0,1,\dots,n-1.
\end{equation} 
It is easy to see that such a polynomial always exists and is unique. The two basic yet fundamental properties of the zeros of orthogonal polynomials $P_n$ are:
\begin{enumerate}[(i)]
    \item All zeros of $P_n$ are real, simple, and belong to the interior of the convex hull of $\supp\,\mu$;
    \item Zeros of two consecutive polynomials $P_n$ and $P_{n+1}$ interlace, that is, between every two consecutive zeros of one of the polynomials there lies exactly one zero of the other polynomial.
\end{enumerate}
For more details on the theory of orthogonal polynomials on the real line, see, e.g.,~\cites{Chihara,Ismail,SimonL2,OPUC1,SzegoBook}. 

\smallskip

Let us now define multiple orthogonal polynomials.  For an introduction to the theory, see \cites{Aptekarev,Ismail,Nikishin,Applications}.
Let $\bm{\mu} = (\mu_1,\dots,\mu_r)$, $r\ge 1$, be a system of (potentially complex) measures with 
all finite moments.
We write $\bm{n}$ for multi-indices $(n_1,\dots,n_r) \in \bbN^r$, and $\abs{\bm{n}} = n_1 + \ldots + n_r$. 

A type I multiple orthogonal polynomial at a location $\bm{n}\in\bbN^r$ is a non-zero vector of polynomials $\bm{A}_{\bm{n}} = (A_{\bm{n}}^{(1)},\dots,A_{\bm{n}}^{(r)})$ such that $\deg{A_{\bm{n}}^{(j)}} \leq n_j - 1$, $j=1,\ldots,r$, and
\begin{equation}\label{eq:moprlI}
	\sum_{j = 1}^r \int {A_{\bm{n}}^{(j)}(x) x^k} d\mu_j(x)
    =
    \begin{cases}
            0, & \mbox{for } k = 0,1,\dots,\abs{\bm{n}}-2, \\
            1, & \mbox{for } k = \abs{\bm{n}}-1
        \end{cases}
\end{equation}
(if $r=1$, then this reduces to $P_{n-1}$ from~\eqref{eq:OPRL} but with a non-monic normalization). 

A type II multiple orthogonal polynomial for $\bm{\mu}$ at the location $\bm{n}$ is a monic polynomial $P_{\bm{n}}(x)$ 
        of exact degree $\abs{\bm{n}}$ such that
	\begin{equation}\label{eq:moprlII}
		\int P_{\bm{n}}(x)x^k d\mu_j(x) = 0,\qquad k = 0,1,\dots,n_j-1 ,\qquad j = 1,\dots,r
	\end{equation} 
(if $r=1$, then this reduces to $P_n$ in~\eqref{eq:OPRL}).
    
We say that an index $\bm{n}$ is normal with respect to $\bm{\mu}$ if $P_{\bm{n}}$ exists and is unique. It is easy to show that this is equivalent to the existence and uniqueness of $\bm{A}_{\bm{n}}$ (see Section~\ref{ss:MOPRLnormality} for details). A system is called perfect if every $\bm{n}\in\bbN^r$ is normal. 

If $r\ge 2$, it is not easy to identify perfect systems, but there are several wide classes of systems which are known to be perfect. These are Angelesco, AT, and Nikishin (see Sections~\ref{ss:Angelesco}, ~\ref{ss:AT}, ~\ref{ss:Nikishin}, for the definitions). In particular, all the known multiple orthogonality analogues of the classical orthogonal polynomials fall within these categories.

We are interested in studying properties of the zeros of type I and type II multiple orthogonal polynomials, specifically their possible locations and interlacing properties. The notion of the (multiple) Christoffel transform is crucial for our analysis. 

Given a measure $\mu$ and a point $z_0\in\bbC$, a one-step Christoffel transform of $\mu$ is the new 
measure $\widehat{\mu}$ defined by 
\begin{equation}\label{eq:Christoffel}
    \int f(x) d\widehat\mu(x) = 
    \int f(x) ({x-z_0}) d\mu(x).
\end{equation}
We denote such a measure by $(x-z_0)\mu$.

Our first main result is the criterion (Theorem~\ref{thm:zerosII}) that says that $z_0$ is a zero of $P_{\bm{n}}$ if and only if the index $\bm{n}$ is not normal for the system
\begin{equation}\label{eq:mutipleChr1}
    ((x-z_0)\mu_1,\dots,(x-z_0)\mu_r).
\end{equation}
Similarly (Theorem~\ref{thm:zerosI}), $z_0$ is a zero of $A^{(1)}_{\bm{n}}$ if and only if the index $\bm{n}$ is not normal for the system
\begin{equation}\label{eq:mutipleChr2}
    ((x-z_0)\mu_1,\mu_2\dots,\mu_r)
\end{equation}
(and similarly for any other $A_{\bm{n}}^{(j)}$). 

Normality of indices for the systems~\eqref{eq:mutipleChr1} and~\eqref{eq:mutipleChr2} are often easy to check, which provides the set free from zeros of the orthogonal polynomials. 
In particular, it can be applied to establish the possible locations of zeros of Angelesco (for type I and type II polynomials), AT (for type II polynomials only, as for type I nothing general can be true), and Nikishin systems (for type I and type II polynomials). All of these results are well known, except for the case of type I polynomials for Nikishin systems (Theorem~\ref{thm:NikishintypeI}), which seems to be new (but see~\cite{LMF} for a related study on real-rootedness). The distinctive feature is that these zeros all appear on a set  disjoint from the supports of the orthogonality measures, which is unlike any other case.


The second result of the paper is the criterion (Theorem~\ref{thm:interlacing type II})  that says that two neighbouring polynomials $P_{\bm{n}}$ and $P_{\bm{n}+\bm{e}_\ell}$ have interlacing zeros
if and only if the index $\bm{n}$ is normal for the system
\begin{equation}
    ((x-z_0)^2\mu_1,\dots,(x-z_0)^2\mu_r)
\end{equation}
for every $z_0\in\bbR$ (assuming real-rootedness of $P_{\bm{n}}$
). We would like to mention that one side of this equivalence was already observed by~\cite{HanVA}, and our proof is in essence distilled from their paper.

A dual result to this is Theorem~\ref{thm:interlacing type I}, which states that $A^{(1)}_{\bm{n}}$ and $A^{(1)}_{\bm{n}-\bm{e}_\ell}$ have interlacing zeros if and only if the index $\bm{n}-2\bm{e}_1$ is normal for the system
\begin{equation}
    ((x-z_0)^2\mu_1,\mu_2\dots,\mu_r)
\end{equation}
for every $z_0\in\bbR$ (assuming real-rootedness of $A_{\bm{n}}^{(1)}$), and similarly for $A_{\bm{n}}^{(j)}$ with any other $j=2,\ldots,r$. 

We then apply these criteria to Angelesco and Nikishin systems to obtain interlacing for both type I and type II polynomials. To the best of our knowledge interlacing for type I neighbouring polynomials for Nikishin systems (Theorem~\ref{thm:NikishintypeInterlacingI}) is new. 

\smallskip

The organization of the paper is as follows. Section~\ref{ss:prelim} contains a collection of results that are needed for the proof. In Section~\ref{ss:MOPRL} we remind the reader the definitions of Angelesco, AT, and Nikishin systems. In Section~\ref{ss:location} we deal with location of zeros and in Section~\ref{ss:interlacing} with interlacing. Section~\ref{ss:Wronskians} has a short discussion on the Christoffel transforms of higher order and their connection to Wronskians of consecutive orthogonal polynomials.

In order to keep the paper self-contained and friendly for the general audience, we present in Section~\ref{ss:MOPRL} elementary proofs of perfectness for Angelesco and AT systems that does not involve zero counting, inspired by the arguments of Kuijlaars~\cite{Kui} and Coussement--Van Assche~\cite{CouVA}. For this we use a generalization of the Andreief identity~\eqref{eq:Andreief} from~\cite{KVMLOPUC}.


Interlacing property of the zeros of the neighbouring type II polynomials for Angelesco systems was shown in~\cite{HanVA} 
(and more generally for any systems with positive $a_{\bm{n},j}$-recurrence coefficients), see also~\cite{ADY20}. Interlacing for type I was proved in~\cite{FMM}, see also~\cite{DenYat}. Interlacing of the zeros of the neighbouring type II polynomials for Nikishin systems was shown in~\cite{FidIllLop04}, and for AT systems in~\cites{Ker70,FidIllLop04,HanVA}. Some other types of interlacing properties of multiple orthogonal polynomials were investigated in~\cites{FLLS,ALR,AKVI,KV1,dosSan,MarMorPer,MarMor24,FMM,Lop21}.


\subsubsection*{Acknowledgements} The authors are grateful to A. Mart\'{i}nez-Finkelshtein for the help with some of the references.

\section{Preliminaries}\label{ss:prelim}
\subsection{A generalization of the Andreief identity}
\hfill\\

In Section~\ref{ss:MOPRL} we use the following generalization of Andreief identity (which is the special case of~\eqref{eq:Andreief} with $M=N$). It is particularly well suited for matrices with block structure which results in simple proofs of perfectness of Angelesco and AT systems. It can also be applied to Nikishin systems~\cite{CouVA,Kui,KNikishin}, as well as Angelesco, AT, and Nikishin systems on the unit circle, see~\cite{KVMLOPUC,KNikishin}.

\begin{prop}[\cite{KVMLOPUC}]\label{prop:Andreief}
    Let $\phi_j,\psi_j \in L^2(\mu)$ for some probability measure on a measure space $(X,\Sigma,\mu)$.  Then for any $N\ge M\ge 1$  and any $(N-M)\times N$ matrix $\bm{A}$:
    \begin{multline}\label{eq:Andreief}
        \det\left(
        \begin{array}{@{}c@{}}
        \bm{A} \\
        \hline
        \begin{pmatrix}
            \int_X \phi_j(x) \psi_k(x)\,d\mu(x)
        \end{pmatrix}_{1\le j \le M, 1\le k\le N}
        \end{array}
        \right) =
        \\
        \tfrac{1}{M!}
         \int_{X^M}
        \det\left(
        \begin{array}{@{}c@{}}
        \bm{A} \\
        \hline 
        \begin{pmatrix}
            \psi_k(x_j)
        \end{pmatrix}_{1\le j \le M, 1\le k\le N}
        \end{array}
        \right)
        \det\left(\phi_l(x_j)\right)_{1\le l,j \le M} \,d^M\mu(\bm{x}),
    \end{multline}
    where $d^M\mu(\bm{x}):=d\mu(x_1)\ldots d\mu(x_M)$. 
\end{prop}

The matrix on the left-hand side of ~\eqref{eq:Andreief} is the $N\times N$ matrix whose upper  $(N-M)\times N$ block is $\bm{A}$ and whose lower
    $M\times N$ block is the matrix $\begin{pmatrix}
            \int_X \phi_j(x) \psi_k(x)\,d\mu(x)
        \end{pmatrix}_{1\le j \le M, 1\le k\le N}$.



\subsection{m-functions and Christoffel 
transforms}\label{ss:mFunctions}
\hfill\\

The $m$-function of $\mu$ is an analytic function on $\bbC\setminus\supp\,\mu$ defined by
\begin{equation}\label{eq:m}
	m_\mu(z) := \int_{\bbR} \frac{d\mu(x)}{x-z}.
\end{equation}
Assuming $\mu$ is a probability measure, $m$ satisfies the following relation, see e.g.,~\cite[Thm 3.2.4]{SimonL2},
\begin{equation}\label{eq:strippingM}
	\frac{1}{m_{\mu}(z)} = b_1-z- a_1^2 m_{\mu^{(1)}}(z),
\end{equation}
for some $a_1,b_1 \in \bbR$ with $a_1>0$, where
\begin{equation}
	m_{\mu^{(1)}}(z) = \int_{\bbR} \frac{d\mu^{(1)}(x)}{x-z} 
\end{equation}
is the $m$-function of another probability measure $\mu^{(1)}$ (which is the spectral measure of the Jacobi operator of $\mu$ but with the first row and column stripped). Using \eqref{eq:strippingM}, it is easy to show that if $\supp\,\mu$ belongs to some interval $\Gamma$ then $\supp\,\mu^{(1)}$ also belongs to $\Gamma$. 

It is an easy exercise to see that the $m$-function ~\eqref{eq:m} $m_{\widehat{\mu}}(z)$ of the Christoffel transform $\widehat{\mu}=(x-z_0)\mu$ (recall~\eqref{eq:Christoffel}) satisfies
\begin{equation}\label{eq:mChristoffel}
    m_{\widehat{\mu}}(z)
    =
    \int \frac{(x-z_0) d\mu(x)}{x-z} = \int d\mu(x) + (z-z_0) m_\mu(z).
\end{equation}



\subsection{Zeros, real-rootedness, and interlacing}
\hfill\\

Given a polynomial $p(z)$, let $\calZ[p]$ stand for the set of its zeros. 
We say that a polynomial is real-rooted if 
$\calZ[p]\subset\bbR$. 

We say that zeros of two real polynomials $p(x)$ and $q(x)$ interlace, and we write $p(x)\sim q(x)$, if 
all zeros of $p(x)$ and $q(x)$ are real, simple, and between every two consecutive zeros of (any) one of the polynomials there lies exactly one zero of the other (here we mean {\it strict} interlacing, i.e., $\calZ[p] \cap \calZ[q] = \varnothing$). 

Recall that a Wronskian of $\ell$ polynomials $Q_1(x),\ldots,Q_\ell(x)$ is 
\begin{equation}\label{eq:wronskian}
    W(Q_1,\ldots,Q_\ell;x):=
        \det \begin{pmatrix}
            Q_{1}(x) & Q_{2}(x) & \ldots & Q_{
            \ell}(x) \\
            Q'_{1}(x) & Q'_{2}(x) & \ldots & Q'_{\ell}(x) \\
            \vdots & \vdots & \ddots & \vdots \\
            Q^{(\ell-1)}_{1}(x) & Q^{(\ell-1)}_{2}(x) & \ldots & Q^{(\ell-1)}_{\ell}(x) 
        \end{pmatrix}.
\end{equation}

To check interlacing we need to following standard result.

\begin{lem}\label{lem:wronskian interlacing}
    Let $Q$ be a real-rooted polynomial and $P$ be any real polynomial. Then $P \sim Q$ if and only if $W(P,Q;z_0) \neq 0$ for all $z_0 \in \bbR$ and $\deg P\le \deg Q+1$. 
\end{lem}
\begin{proof}
    

    Without loss of generality we may assume $P$ and $Q$ to be monic. Suppose $W(P,Q;x) \neq 0$ on $\bbR$, with $\deg Q = N$, and $\deg P\le N+1$. Since $W(P,Q)$ is continuous it keeps the same sign for all $x \in \bbR$. If $x_j$ is a zero of $Q$, we have $W(P,Q;x_j) = P(x_j)Q'(x_j)$, so all zeros of $Q$ are simple and different from the zeros of $P$. We write $x_1>x_2>\ldots>x_N \in \bbR$ for the zeros of $Q$.
    
    Since $Q'(x_j)$ alternates sign, we obtain that between each two consecutive zeros of $Q$ there must be an odd number of zeros of $P$ (counting multiplicities). Switching the roles of $P$ and $Q$, we obtain that there is exactly one zero of $P$ on each interval $(x_j,x_{j+1})$. This determines the location of $N-1$ zeros of $P$ and proves $P\sim Q$ unless  $\deg P=N+1$. If $\deg P= N+1$ then we claim there is one zero of $P$ on $(x_1,\infty)$. If not then $P(x_1)>0$ and therefore $W(P,Q;x_1) = P(x_1) Q'(x_1)>0$. But  this contradicts $W(P,Q;x)\sim x^{N+1}(N x^{N-1})-(n+1) x^N x^N = -x^{2N}$ as $x\to\infty$. Similarly one shows that $P$ must have a zero on $(-\infty,x_N)$. This completes the proof of $P\sim Q$.

    For the converse, 
    assume $\deg{Q} \geq \deg{P}$ and write 
    \begin{equation}
        W(P,Q;x) = -Q(x)^2(P/Q)'(x) = Q(x)^2\sum_{j = 1}^{N}\frac{\lambda_j}{(x-x_j)^2},
    \end{equation}
    where $\lambda_j$ is the residue of $P/Q$ at $x_j$. 
    Interlacing $P\sim Q$ implies that all $\lambda_j$'s are of the same sign, and then the last equality proves that $W(P,Q;x)$ does not vanish on $\bbR$.
    The case $\deg Q \le \deg P$ can be handled by swapping the roles of $P$ and $Q$.
\end{proof}

\section{Multiple orthogonal polynomials}\label{ss:MOPRL}
\subsection{Normality}\label{ss:MOPRLnormality}
\hfill\\



Let $r\ge 1$ and consider a system of measures  on the real line $\bm{\mu} = (\mu_1,\dots,\mu_r)$. 
$\bm{\mu} = (\mu_1,\dots,\mu_r) \in \mathcal{M}(\bbR)^r$. 
Recall the definition of the type I and type II multiple orthogonal polynomials~\eqref{eq:moprlI}, \eqref{eq:moprlII}.

It is not hard to verify that existence and uniqueness of either type I or type II polynomials is equivalent to the condition $\det H_{\bm{n}}[\bm{\mu}]\ne 0$, where
\begin{equation}
	\label{eq:D}
	H_{\bm{n}} [\bm{\mu}]= 
  \begin{pmatrix}
c_0^{(1)} & c_1^{(1)} & \cdots & c_{\abs{\bm{n}}-1}^{(1)} \\
c_1^{(1)} & c_2^{(1)} & \cdots & c_{\abs{\bm{n}}}^{(1)} \\
\vdots & \vdots & \ddots & \vdots \\
c_{n_1-1}^{(1)} & c_{n_1}^{(1)} & \cdots & c_{\abs{\bm{n}}+n_1-2}^{(1)} \\
\hline
& & \vdots \\
\hline
c_0^{(r)} & c_1^{(r)} & \cdots & c_{\abs{\bm{n}}-1}^{(r)} \\
c_1^{(r)} & c_2^{(r)} & \cdots &  c_{\abs{\bm{n}}}^{(r)} \\
\vdots & \vdots & \ddots & \vdots \\
c_{n_r-1}^{(r)} & c_{n_r}^{(r)} & \cdots & c_{\abs{\bm{n}}+n_r-2}^{(r)} \\
\end{pmatrix}.
\end{equation}
Here 
$c^{(j)}_k = \int x^k d\mu_j(x)$.
If this happens then we say that $\bm{n}$ is normal with respect to $\bm{\mu}$. $\bm{\mu}$ is perfect if every $\bm{n}\in\bbN^r$ is normal.




\begin{rem}\label{rem:normality} It is easy to see that $\bm{n}$ is normal if and only if there is no non-zero solution to \eqref{eq:moprlII} with $\deg{P_{\bm{n}}} < \abs{\bm{n}}$. Similarly, $\bm{n}$ is normal if and only if there is no non-zero vector $\bm{A}_{\bm{n}}$ with $\deg{A_{\bm{n}}^{(j)}} \leq n_j - 1$ and 
\begin{equation}
    \sum_{j = 1}^r\int A_{\bm{n}}^{(j)}(x)x^k d\mu_j(x) = 0, \qquad k = 0,\dots,\abs{\bm{n}}-1.
\end{equation}
\end{rem}

\subsection{Angelesco systems and their perfectness}\label{ss:Angelesco}
\hfill\\

\begin{defn}\label{def:Angelesco}
    For each $1\le j \le r$, let $\mu_j\in\mathcal{M}(\bbR)$. We call $(\mu_j)_{j=1}^r$ an Angelesco system if there exist intervals $\Gamma_j$, $1\le j \le r$, such that $\supp\,\mu_j\subseteq \Gamma_j$ and for each $j\ne k$, $\Gamma_j\cap \Gamma_k$ is either empty or consists of a single point.
\end{defn}

\begin{thm}\label{thm:Angelesco}
    Angelesco systems are perfect.
\end{thm}
\begin{proof}
Write $H^{(j)}_{n_j,|\bm{n}|}$ for the $j$-th block in \eqref{eq:D}. Apply~\eqref{eq:Andreief} $r$ times to each of the $H^{(j)}_{n_j,|\bm{n}|}$ sequentially (use $N=|\bm{n}|$ and $M=n_j$) and use the disjointness of the supports of $\mu_j$ to obtain
\begin{multline*}
    \det H_{\bm{n}} = \frac{1}{n_1!\ldots n_r!}\int_{\Gamma_1^{n_1}}\!\!\ldots\! \int_{\Gamma_r^{n_r}} \Delta_{|\bm{n}|}(\bm{x}_1,\ldots,\bm{x}_r)\prod_{j=1}^r \Delta_{n_j}(\bm{x}_j) \, d^{n_j} \mu_j(\bm{x}_j),
\end{multline*}
where $\bm{x}_j = (x_{j,1},\dots,x_{j,n_j})$, and $\Delta_k(x_1,\dots,x_k) = \prod_{i<j}(x_j-x_i)$ is the $k\times k$ Vandermonde determinant. We end up with the integral
\begin{equation}
    \frac{1}{n_1!\ldots n_r!}\int_{\Gamma_1^{n_1}}\!\!\ldots\! \int_{\Gamma_r^{n_r}}
    \prod_{i<j} \Delta(\bm{x}_i,\bm{x}_j)
    \prod_{j=1}^r \Delta(\bm{x_j})^2 \, d^{n_j} \mu_j(\bm{x}_j),
\end{equation}
where $\Delta(\bm{x}_i,\bm{x}_j) := \prod_{\alpha,\beta} (x_{j,\alpha}-x_{i,\beta})$. Clearly the integrand does not change sign, so cannot be zero, and we conclude that $\bm{n}$ is normal. 
\end{proof}

\begin{rem}\label{rem:signAngelesco}
    Assuming intervals $\Gamma_j$ are ordered according to $\Gamma_i\le \Gamma_j$ for all $i<j$, the above proof also shows that $\det H_{\bm{n}} > 0$. {This fact can be used to show that all the nearest-neighbour recurrence coefficients~\cite{NNRR} $a_{\bm{n},j}$ for Angelesco systems are positive}.
\end{rem}

\subsection{AT systems and their perfectness}\label{ss:AT}
\hfill\\

Let $\Gamma$ be a closed interval of $\bbR$. 
 A collection $(u_j(t))_{j=1}^n$ of continuous real-valued functions on 
 $\Gamma$ is called a Chebyshev system on $\Gamma$ if the determinant
    \begin{equation}\label{eq:Chebyshev}
        U_n(\bm{x}):=\det
        \begin{pmatrix}
            u_1(x_1) & u_1(x_2) & \cdots & u_1(x_n) \\
            u_2(x_1) & u_2(x_2) & \cdots & u_2(x_n) \\
            \vdots & \vdots & \ddots & \vdots \\
            u_n(x_1) & u_n(x_2) & \cdots & u_n(x_n) \\
        \end{pmatrix} 
    \end{equation}
    is non-zero and has a constant sign
    for any $x_1<x_2<\ldots <x_n$ on $\Gamma$. See~\cite{KarStu} to more details about the Chebyshev property and its equivalent form.  
    


    \begin{defn}\label{def:AT}
        For each $1\le j \le r$, let $d\mu_j(x)=w_j(x)d\mu(x)$ be a  measure that is absolutely continuous with respect to some $\mu\in\mathcal{M}({\Gamma})$. We call $(\mu_j)_{j=1}^r$ an AT system on $\Gamma$ for the multi-index $\bm{n} = (n_1,\ldots,n_r)$ if the functions
        \begin{equation}\label{eq:AT}
	\{w_1, xw_1,\ldots,x^{n_1-1}w_1,w_2,xw_2,\ldots,x^{n_2-1}w_2,\ldots,w_r,xw_r,\ldots,x^{n_r-1}w_r\}
\end{equation}
        form a Chebyshev system on $\Gamma$.
    \end{defn}

    \begin{thm}\label{thm:ATperfect}
        If $(\mu_j)_{j=1}^r$ is AT for the index $\bm{n}$ then $\bm{n}$ is normal.
    \end{thm}
    \begin{proof}
        Applying the Andreief identity (\eqref{eq:Andreief} with $M=N=|\bm{n}|$) and performing elementary row operations we arrive to
        \begin{equation}
            \det H_{\bm{n}}[\bm{\mu}] = \frac{1}{|\bm{n}|!} \int_{{\Gamma}^{|\bm{n}|}} U_{|\bm{n}|}(\bm{x}) \Delta_{|\bm{n}|}(\bm{x})  \, d^{|\bm{n}|} \mu(\bm{x}),
        \end{equation}
        where $W$ is~\eqref{eq:Chebyshev} for the system~\eqref{eq:AT}, and $\Delta$ is again the Vandermonde determinant. Since both $U_{|\bm{n}|}$ and $\Delta_{|\bm{n}|}$ preserve sign on $x_1<\ldots<x_{|\bm{n}|}$, and simultaneously change sign when we permute the $x_j$'s, we see that the integrand does not change sign. Hence $\det H_{\bm{n}}\ne 0$, so that $\bm{n}$ is normal. 
    \end{proof}


\begin{rem}
    Similarly to Remark~\ref{rem:signAngelesco}, we can see that $\operatorname{sgn} \det H_{\bm{n}} =\operatorname{sgn} U_{\bm{n}}$. The collection of signs of $U_{\bm{n}}$ therefore determines the signs of the nearest-neighbour recurrence coefficients $a_{\bm{n},j}$ for AT systems.
\end{rem}

\subsection{Nikishin systems}\label{ss:Nikishin}
\hfill\\



For any set $S\subseteq\bbR$, 
we denote  $\mathring{S}$ to be the interior of $S$ in the topology of $\bbR$. 
Given two closed intervals $\Gamma_1$ and $\Gamma_2$ with $\mathring{\Gamma}_1\cap\mathring{\Gamma}_2=\varnothing$, and two measures $\sigma_1\in\calM(\Gamma_1)$,  $\sigma_2\in\calM(\Gamma_2)$, denote $\la  \sigma_1,\sigma_2 \ra$ to be the measure in $\calM(\Gamma_1)$ given by
\begin{equation}\label{eq:brackets}
	d \la  \sigma_1,\sigma_2 \ra (x) := m_{\sigma_2}(x) d\sigma_1(x), 
\end{equation}
where $m_{\sigma_2}$ is the $m$-function of $\sigma_2$ (see Section \ref{ss:mFunctions}). Note that 
the right-hand side of~\eqref{eq:brackets} always defines a finite sign-definite measure on $\Gamma_1$ assuming that $\Gamma_1$ and $\Gamma_2$ are disjoint.
If $\Gamma_1$ and $\Gamma_2$ share a common endpoint however, then by writing $\la  \sigma_1,\sigma_2 \ra$, we implicitly {\it assume} that ~\eqref{eq:brackets} defines a finite measure, as this is no longer automatic and depends on the behavior of $\sigma_1$ and $\sigma_2$ at that point. 

\begin{defn}\label{def:Nikishin}
We say that $\bm{\mu}=(\mu_1,\ldots,\mu_r)\in\calM(\bbR)^r$ forms a Nikishin system generated by $(\sigma_1,\ldots,\sigma_r)$ (we then write $\bm{\mu}=\calN(\sigma_1,\ldots,\sigma_r)$), if there is
a collection of closed intervals $\Gamma_j$, $j=1,\ldots,r$ such that
\begin{equation}
	\mathring{\Gamma}_j \cap \mathring{\Gamma}_{j+1} = \varnothing, \qquad j=1,\ldots,r-1,
\end{equation}
and measures $\sigma_j \in\calM(\Gamma_j)$, $j=1,\ldots,r$, so that
\begin{equation}
	\mu_1 = \sigma_1,  \mu_{2} = \la \sigma_1, \sigma_{2}\ra, \mu_{3} = \la \sigma_1, \la \sigma_{2},\sigma_3\ra \ra, \ldots, \mu_r = \la \sigma_1, \la \sigma_{2},\la \sigma_3,\ldots,\sigma_r\ra\ra \ra.
\end{equation}
\end{defn}


It is well-known that Nikishin systems are perfect. The proof for $r=2$ 
appeared in~\cite{DriSta94}, and for any $r\ge 2$ this was shown in~\cite{FidLop11,FidLop11b}. See also~\cite{Nik80,BBFL,FidIllLop04,FidLop02,CouVA,Kui,Lop21} for related results. 


\section{Zero location}\label{ss:location}
\subsection{Zero location for Type II multiple orthogonal polynomials}\label{ss:locationII}
\hfill\\

The following theorem provides a very simple yet useful tool for locating zeros of the type II multiple orthogonal polynomials. Recall that the notation $(x-z_0)\mu$ stands for the Christoffel transform of $\mu$, see Section~\ref{ss:mFunctions}. Let us also write $(x-z_0)\bm{\mu}$ for  $((x-z_0)\mu_1,\dots,(x-z_0)\mu_r)$.

\begin{thm}\label{thm:zerosII}
    Let $\bm{n}$ be normal with respect to $\bm{\mu}$. Then 
    \begin{equation}
        \calZ[P_{\bm{n}}] = \left\{
        z_0\in\bbC: \bm{n} \mbox{ is not normal for } (x-z_0)\bm\mu \right\}.
    \end{equation}
\end{thm}
\begin{rem}\label{rem:multiplicityII}
    The same proof shows that $z_0$ being a zero of $P_{\bm{n}}$ of multiplicity $\ge 2$ implies that $\bm{n}$ is not normal for $(x-z_0)^2 \bm\mu$.
\end{rem}
\begin{rem}\label{rem:complexZerosII}
    Suppose each $\mu_j\in\mathcal{M}(\bbR)$ and let $z_0\in\bbC\setminus\bbR$. Then the same proof shows that $P_{\bm{n}}(z_0)=0$ implies that $\bm{n}$ is not normal for $|x-z_0|^2 \bm\mu$. 
\end{rem}
\begin{proof}
    If $P_{\bm{n}}(z_0) = 0$ then $Q(x) = P_{\bm{n}}(x)/(x-z_0)$ satisfies every orthogonality relation for the index $\bm{n}$ with respect to $(x-z_0)\bm{\mu}$. Since $\deg{Q} < \abs{\bm{n}}$, we obtain that $\bm{n}$ is not normal for $(x-z_0)\bm{\mu}$ by Remark \ref{rem:normality}. 

    Conversely, if $\bm{n}$ is not normal for $(x-z_0)\bm{\mu}$ then by Remark \ref{rem:normality} there is some polynomial $Q \neq 0$ with $\deg{Q} < \abs{\bm{n}}$ and
    \begin{equation}
        \int Q(x)x^k (x-z_0)d\mu_j(x) = 0, \qquad k = 0,\dots,n_j-1.
    \end{equation}
    Since $\bm{n}$ is normal, we then necessarily have $P_{\bm{n}}(x) = Q(x)(x-z_0)$, so $P_{\bm{n}}(z_0) = 0$. 
\end{proof}

\begin{cor}\label{cor:ATtypeII}
    Let $\bm{\mu}$ be an AT system on a closed 
    interval 
    $\Gamma$ 
    for an index $\bm{n}\in\bbN^r$. Then $\calZ[P_{\bm{n}}]\subset \mathring{\Gamma}$, and each zero is simple.
\end{cor}
\begin{proof}
    Let $z_0\in\bbR\setminus \mathring{\Gamma}$. $\bm{\mu}$ being AT on $\Gamma$ at $\bm{n}$ implies that $(x-z_0)\bm{\mu}$ is also AT on $\Gamma$ at $\bm{n}$: just replace the reference measure $\mu$ in Definition~\ref{def:AT} with $(x-z_0)\mu$.
    By Theorems~\ref{thm:ATperfect} and~\ref{thm:zerosII}, we get that $z_0$ is not a zero of $P_{\bm{n}}$. 

    Let $z_0\in\bbC\setminus \bbR$. Again, it is elementary to see that $|x-z_0|^2 \bm\mu$ is AT on $\Gamma$ for $\bm{n}$. By Theorem~\ref{thm:ATperfect} and Remark~\ref{rem:complexZerosII}, we get that $z_0$ is not a zero of $P_{\bm{n}}$. 

    That each zero must be simple follows from Remark~\ref{rem:multiplicityII} and AT property of $(x-z_0)^2\bm{\mu}$ on $\Gamma$ for $\bm{n}$ for any $z_0\in\bbR$.
\end{proof}
\begin{rem}
    An analogous statement holds for any Nikishin system by the same argument.
\end{rem}

Note that if $\bm{\mu}$ is an Angelesco system then both $(x-z_0)\bm{\mu}$ (for $z_0\notin \bbR\setminus \bigcup_{j=1}^r \mathring{\Gamma}_j$) and $|x-z_0|^2 \bm\mu$ (for $z_0\in\bbC\setminus \bbR$) are also Angelesco. Therefore applying the exact same argument as for AT, one obtains $\calZ[P_{\bm{n}}]\subset \bigcup_{j=1}^r \mathring{\Gamma}_j$ with each zero simple. This statement can be improved using the following well-known result. 





\begin{prop}\label{prop:quasiorthogonal}
    Let $\Gamma$ be an interval in $\bbR$ and $\mu\in\mathcal{M}(\Gamma)$. If  a polynomial $P(x)$ satisfies 
    \begin{equation}\label{eq:quasiorthogonal}
         \int P(x) x^k d {\mu}(x) = 0, \qquad k = 0,\dots,n-1,
     \end{equation}
     then it has at least $n$ zeros on $\mathring{\Gamma}$.
    \end{prop}
    \begin{proof}
        Let $Q(x)$ be the monic polynomial whose zeros are the zeros of $P$ that lie in $\mathring{\Gamma}$. Then the Christoffel transform $\widehat{\mu}:=\frac{P(x)}{Q(x)} \mu$ of $\mu$ belongs to $\mathcal{M}(\Gamma)$, and therefore it has a unique orthogonal polynomial of degree $n$. But~\eqref{eq:quasiorthogonal} implies
        \begin{equation}
        \int Q(x) x^k d\widehat{\mu}(x) = 0, \qquad k = 0,\dots,n-1.
        \end{equation}
        The $n$-th orthogonal polynomial of $\widehat{\bm{\mu}}$ cannot have degree less than $n$, which implies $\deg Q \ge n$.
    \end{proof}


\begin{cor}\label{cor:angelesco zero location}
    Let $\bm{\mu}$ be an Angelesco system and $\bm{n}\in\bbN^r$. Then $P_{\bm{n}}$ has exactly $n_j$ zeros on each $\mathring{\Gamma}_j$, and each zero is simple.
 \end{cor}
\begin{proof}
    Since $\bm{n}$ is normal, we have $\deg P_{\bm{n}} =|\bm{n}|$. By Proposition~\ref{prop:quasiorthogonal}, $P_{\bm{n}}$ has at least $n_j$ zeros on each $\mathring{\Gamma}_j$. Since $\mathring{\Gamma}_j$'s are pair-wise disjoint, we obtain the statement.
\end{proof}


\subsection{Zero location for Type I multiple orthogonal polynomials}\label{ss:locationI}
\hfill\\

\begin{thm}\label{thm:zerosI}
    Let $\bm{n}$ be normal with respect to $\bm{\mu}$. Then 
    \begin{equation}
        \calZ[A^{(1)}_{\bm{n}}] = \left\{
        z_0\in\bbC: \bm{n} -\bm{e}_1 \mbox{ is not normal for } ((x-z_0)\mu_1,\mu_2,\ldots,\mu_r) \right\}.
    \end{equation}
    Analogous statements work for $A_{\bm{n}}^{(j)}$ with any $j$.
\end{thm}
\begin{rem}\label{rem:multiplicityI}
    The same proof shows that if $z_0$ is a zero of $A^{(1)}_{\bm{n}}$ of multiplicity $\ge 2$ then $\bm{n}-2\bm{e}_1$ and $\bm{n}-\bm{e}_1$ are not normal for $((x-z_0)^2 \mu_1,\mu_2,\ldots,\mu_r)$. Analogous statements work for $A_{\bm{n}}^{(j)}$ with any $j$.
\end{rem}
\begin{rem}\label{rem:complexZerosI}
    Suppose each $\mu_j\in\mathcal{M}(\bbR)$ and let $z_0\in\bbC\setminus\bbR$. Then the same proof shows that $A^{(1)}_{\bm{n}}(z_0)=0$ implies that $\bm{n}-2\bm{e}_1$ and  $\bm{n}-\bm{e}_1$ are not normal for $(|x-z_0|^2 \mu_1,\mu_2,\ldots,\mu_r)$. Analogous statements work for $A_{\bm{n}}^{(j)}$ with any $j$.
\end{rem}
\begin{proof}
    Let us denote $((x-z_0)\mu_1,\mu_2,\ldots,\mu_r)$ by $\widehat{\bm{\mu}}=(\widehat{\mu}_1,\ldots,\widehat{\mu}_r)$. Since $\bm{n}$ is $\bm{\mu}$-normal, we have
    \begin{equation}\label{eq:typeIlocationProof}
    \sum_{j = 1}^r\int A_{\bm{n}}^{(j)}(x) x^k d\mu_j(x) = 0, \qquad k = 0,\dots,\abs{\bm{n}}-2.
    \end{equation}
    If $A_{\bm{n}}^{(1)}(z_0) = 0$, then let $\bm{B}(x)$ be  $(A_{\bm{n}}^{(1)}(x)/(x-z_0),A_{\bm{n}}^{(2)}(x),\dots,A_{\bm{n}}^{(r)}(x))$. By~\eqref{eq:typeIlocationProof},
    \begin{equation}\label{eq:typeIlocationProof2}
    \sum_{j = 1}^r\int B^{(j)}(x) x^k d\widehat{\mu}_j(x) = 0, \qquad k = 0,\dots,\abs{\bm{n}}-2.
    \end{equation}
    Since $\deg{B^{(1)}} \leq n_1-2$ and $\deg{B^{(j)}} \leq n_j-1$ for $j = 2,\dots,r$, Remark \ref{rem:normality} shows that $\bm{n}-\bm{e}_1$ is not $\widehat{\bm{\mu}}$-normal. 
    
    Conversely, if $\bm{n}-\bm{e}_1$ is not normal for $\widehat{\bm{\mu}}$ then by Remark \ref{rem:normality} there is a vector $\bm{B} = (B^{(1)},\dots,B^{(r)}) \neq \bm{0}$ with $\deg{B^{(1)}} \leq n_1-2$ and $\deg{B^{(j)}} \leq n_j-1$ for $j = 2,\dots,r$, that satisfies~\eqref{eq:typeIlocationProof2}. Then $\bm{A}_{\bm{n}}(x)$ must be equal to $(B^{(1)}(x)(x-z_0),B^{(2)}(x),\dots,B^{(r)}(x))$, up to a multiplicative normalization, so $A_{\bm{n}}^{(1)}(z_0) = 0$.


    As for the remarks, assume that $A_{\bm{n}}^{(1)}(z_0) = 0$ of multiplicity $\ge 2$. Denote $((x-z_0)^2\mu_1,\mu_2,\ldots,\mu_r)$ by $\widehat{\widehat{\bm{\mu}}}=(\widehat{\widehat{\mu}}_1,\ldots,\widehat{\widehat{\mu}}_r)$ and let $\bm{B}(x)$ be the system $(A_{\bm{n}}^{(1)}(x)/(x-z_0)^2,A_{\bm{n}}^{(2)}(x),\dots,A_{\bm{n}}^{(r)}(x))$. By~\eqref{eq:typeIlocationProof},
    \begin{equation}\label{eq:typeIlocationProof2}
    \sum_{j = 1}^r\int B^{(j)}(x) x^k d\widehat{\widehat{\mu}}_j(x) = 0, \qquad k = 0,\dots,\abs{\bm{n}}-2.
    \end{equation}
    Since $\deg{B^{(1)}} \leq n_1-3$ and $\deg{B^{(j)}} \leq n_j-1$ for $j = 2,\dots,r$, Remark \ref{rem:normality} shows that both $\bm{n}-2\bm{e}_1$ and $\bm{n}-\bm{e}_1$ are not $\widehat{\bm{\mu}}$-normal. 
\end{proof}

\begin{cor}
    Let $\bm{\mu}$ be an Angelesco system. Then $\calZ[A^{(j)}_{\bm{n}}]\subset \mathring{\Gamma}_j$ for any $j=1,\ldots,r$, and each zero is simple.
\end{cor}
\begin{proof}
    For any $j$ and any $z_0\in\bbR\setminus \mathring{\Gamma}_j$, replacing $\mu_j$ with $(x-z_0)\mu_j$ leads to another Angelesco system. By Theorems~\ref{thm:Angelesco} and~\ref{thm:zerosI}, we get that $z_0$ is not a zero of $A^{(j)}_{\bm{n}}$. 

    Let $z_0\in\bbC\setminus \bbR$. Then Remark~\ref{rem:complexZerosI} together with the Angelesco property of $(|x-z_0|^2 \mu_1,\mu_2,\ldots,\mu_r)$ shows  that $z_0$ is not a zero of any $A^{(j)}_{\bm{n}}$. 

    Simplicity of zeros follows from Remark~\ref{rem:multiplicityI} and the Angelesco property of $((x-z_0)^2 \mu_1,\mu_2,\ldots,\mu_r)$ for $z_0\in\bbR$.
\end{proof}

\subsection{Zero location in Nikishin Systems}\label{ss:locationINikishin}
\hfill\\

Note that if $\bm\mu$ is AT on $\Gamma$, then if we replace $\mu_j$ with $(x-z_0)\mu_j$ with $z_0\in\bbR\setminus\Gamma$, then the new system need not satisfy the AT property anymore. This shows that one should not expect any general results about the zeros of type I multiple polynomials for AT systems. This should not be surprising, since it is known that they do not even have to be real-rooted in general. We can say something about zeros of type I polynomials for Nikishin system, however. For simplicity we only consider the case $r=2$. We will employ the following simple result.

\begin{lemma}\label{lem:perturbation}
	Let $\bm{\mu}:=(\mu_1,\mu_2)$ and $\widetilde{\bm{\mu}}:=(\mu_1,\widetilde{\mu}_2)$, where
	\begin{equation}
		d\widetilde{\mu}_2(x) = d\mu_2(x) + Q(x) d\mu_1(x), \quad Q(x) = \sum_{j=0}^s k_j x^j.
	\end{equation}
	Then for any $\bm{n}=(n_1,n_2)\in\bbN^2$ with $n_2 \le n_1-s$, 
    $\bm{n}$ is normal with respect to $\bm{\mu}$ if and only if it is normal with respect to $\widetilde{\bm{\mu}}$. In this case,
    \begin{align}
        \widetilde{P}_{\bm{n}}(x) &= P_{\bm{n}}(x), \\
        \widetilde{A}_{\bm{n}}^{(1)}(x) &= A_{\bm{n}}^{(1)}(x) - Q(x)A_{\bm{n}}^{(2)}(x),\\
        \widetilde{A}_{\bm{n}}^{(2)}(x) &= A_{\bm{n}}^{(2)}(x),
    \end{align}
    where $\widetilde{P}_{\bm{n}}$ are the type II and $\bm{\widetilde{A}}_{\bm{n}} = (\widetilde{A}_{\bm{n}}^{(1)},\widetilde{A}_{\bm{n}}^{(2)})$ are the type I polynomials with respect to $\widetilde{\bm{\mu}}$.
\end{lemma}
\begin{proof}
	Observe that $H_{n_1,|\bm{n}|}^{(1)}$ block is common for both ${H}_{\bm{n}}[\bm\mu]$ and ${{H}}_{\bm{n}} [\widetilde{\bm{\mu}}]$. Now note that the $k$-th row of  $\widetilde{H}_{n_2,|\bm{n}|}^{(2)}$ ($k=1,\ldots,n_2$) is 
	\begin{multline}
		\big(\widetilde{c}_{k-1}^{(2)}, \widetilde{c}_{k}^{(2)},\ldots, \widetilde{c}_{k-2+|\bm{n}|}^{(2)}\big)
		=
		\big({c}_{k-1}^{(2)}, {c}_{k}^{(2)},\ldots, {c}_{k-2+|\bm{n}|}^{(2)}\big)
		\\
		+\sum_{j=0}^s k_j
		\big({c}_{k+j-1}^{(1)}, {c}_{k+j}^{(1)},\ldots, {c}_{k+j-2+|\bm{n}|}^{(1)}\big).
	\end{multline}
	If $n_2 \le n_1-s$, each row of moments in the last sum can be canceled in $\det {{H}}_{\bm{n}} [\widetilde{\bm{\mu}}]$ by subtracting a multiple of the corresponding row of $H_{n_1,|\bm{n}|}^{(1)}$ (use $n_2 \le n_1-s$). This shows that $\det {{H}}_{\bm{n}} [\widetilde{\bm{\mu}}]$  is equal  to $\det {{H}}_{\bm{n}} [{\bm{\mu}}]$.

    Now we write
    \begin{align}
        & \int \widetilde{A}_{\bm{n}}^{(1)}(x)x^p d{\mu}_1(x) + \int \widetilde{A}_{\bm{n}}^{(2)}(x)x^pd\widetilde{\mu}_2(x) \\ & = \int \big(\widetilde{A}_{\bm{n}}^{(1)}(x) 
+ Q(x)\widetilde{A}_{\bm{n}}^{(2)}(x)\big)x^p d\mu_1(x) + \int \widetilde{A}_{\bm{n}}^{(2)}(x)x^pd\mu_2(x).
    \end{align}
    If $n_2\le n_1-s$ then $\deg{QA_{\bm{n}}^{(2)}} \leq n_1-1$. Hence by normality of $\bm{n}$ for $\bm{\mu}$, we see that $\bm{\widetilde{A}}_{\bm{n}} = (A_{\bm{n}}^{(1)} - QA_{\bm{n}}^{(2)},A_{\bm{n}}^{(2)})$. The proof of $\widetilde{P}_{\bm{n}} = P_{\bm{n}}$ is similar.
\end{proof}

\begin{thm}\label{thm:NikishintypeI}
    Let $(\mu_1,\mu_2)=\calN(\sigma_1,\sigma_2)$ be a Nikishin system. 
    \begin{enumerate}[(i)]
        \item Let $\bm{n}\in\bbN^2$ with $n_1 +1 \le n_2$, $n_1\ne 0$. Then $A^{(1)}_{\bm{n}}$ is real-rooted and 
        \begin{equation}\label{eq:NikishinTypeI1}
        \calZ[A^{(1)}_{\bm{n}}]\subset \mathring{\Gamma}_2,
        \end{equation}
        and each zero is simple.
        \item Let $\bm{n}\in\bbN^2$ with $n_1+1 \ge n_2$, $n_2 \ne 0$. Then $A^{(2)}_{\bm{n}}$ is real-rooted and  
        \begin{equation}\label{eq:NikishinTypeI2}
            \calZ[A^{(2)}_{\bm{n}}]\subset \mathring{\Gamma}_2,
        \end{equation}            
        and each zero is simple.
    \end{enumerate}
\end{thm}
\begin{rem}\label{rem:simulationsNikishin}
    Our simulations clearly show that 
    the above statement is optimal  in the sense that ~\eqref{eq:NikishinTypeI1} does not have to hold if $n_1+1>n_2$ and~\eqref{eq:NikishinTypeI2} does not have to hold if $n_1+1<n_2$. However, real-rootedness for $A^{(1)}_{\bm{n}}$ also holds if $n_1=n_2$, and  for $A^{(2)}_{\bm{n}}$ if $n_1=n_2-2$, see Remark~\ref{rem:oneMoreDiagonal}. Beyond these locations, one typically should expect an appearance of complex zeros. 
\end{rem}
\begin{rem}    
    Theorem~\ref{thm:NikishintypeI} and Remark  ~\ref{rem:oneMoreDiagonal} show that all the type I polynomials are real-rooted along the step-line indices (which are the indices of the form $(n,n)$ and $(n,n+1)$).
\end{rem}
\begin{proof}
   \textit{(ii)} Let $z_0 \in \bbR\setminus \mathring{\Gamma}_2$.
   Note that 
    \begin{equation}\label{eq:proofNikZero}
        (\mu_1,(x-z_0)\mu_2) = (\mu_1,(x-z_0)m_{\sigma_2}(x)\mu_1).
        \end{equation}
        By~\eqref{eq:mChristoffel}, 
        \begin{equation}
            (x-z_0) m_{\sigma_2}(x) = m_{\widehat{\sigma}_2}(x) - c
        \end{equation}
        for some $c\in\bbR$, which means~\eqref{eq:proofNikZero} can be represented as
        \begin{equation}\label{eq:proofNikZero2}
            (\mu_1,m_{\widehat{\sigma}_2}(x) \mu_1 - c\mu_1).
        \end{equation}     
        Note that $\supp \,\widehat{\sigma}_2 \subseteq \supp\,{\sigma_2}\subseteq \Gamma_2$, so that $(\mu_1,m_{\widehat{\sigma}_2}(x) \mu_1)$ is Nikishin, and therefore it has every index normal. 
   Applying Lemma~\ref{lem:perturbation} (with $s=0$), we obtain that ~\eqref{eq:proofNikZero2} has every $(n_1,n_2)$ with $n_2\le n_1$  normal. Then  Theorem~\ref{thm:zerosI} proves that $z_0$ is not a zero of $A^{(2)}_{\bm{n}}$ for any $(n_1,n_2)$ with $n_2-1\le n_1$.


         Now let $z_0\in\bbC\setminus\bbR$. Consider $\widetilde{\mu}_2 = \abs{x-z_0}^2\mu_2 = (x-z_0)(x-\bar{z}_0)m_{\sigma_2}(x) \mu_1$. Apply \eqref{eq:mChristoffel} twice to get
         \begin{equation}\label{eq:doubleChristoffel}
             (\mu_1,\widetilde{\mu}_2) = (\mu_1,m_{\widetilde{\sigma}_2}(x)\mu_1 + (cx+d)\mu_1)
         \end{equation}
         for some $c \in \bbR$ and $d \in \bbC$. Since $(\mu_1,\widetilde{\mu}_2)$ is Nikishin, it is perfect. Applying Lemma \ref{thm:zerosI} (with $s=1$), we obtain that~\eqref{eq:doubleChristoffel} has every $(n_1,n_2)$ with $n_2\le n_1-1$ normal. Then Remark~\ref{rem:complexZerosI} shows that $z_0$ is not a zero of $A^{(2)}_{\bm{n}}$ for any $(n_1,n_2)$ with $n_2-2\le n_1-1$. This completes the proof of~\eqref{eq:NikishinTypeI2}. To show simplicity of zeros, apply the exact same argument but with Remark~\ref{rem:multiplicityI} instead of Remark~\ref{rem:complexZerosI}.
         

          \textit{(i)} 
          Let $\widetilde{\bm{\mu}} = (\mu_2,\mu_1) = (\widetilde{\mu}_1,m_{\sigma_2}(x)^{-1}\widetilde{\mu}_1)$, where $\widetilde{\mu}_1 = m_{\sigma_2}(x)\mu_1$. Note that $A_{n_1,n_2}^{(1)} = \widetilde{A}_{n_2,n_1}^{(2)}$. By the stripping relation~\eqref{eq:strippingM}, we have
          \begin{equation}
              \widetilde{\bm{\mu}} = (\widetilde{\mu}_1,(b-x)\widetilde{\mu}_1-am_{{\sigma}^{(1)}_2}\widetilde{\mu}_1).
          \end{equation}
          By Lemma \ref{lem:perturbation}, we then have that $(\widetilde{A}_{n_2,n_1}^{(1)},\widetilde{A}_{n_2,n_1}^{(2)})$ are the type I polynomials for $(n_2,n_1)$ with respect to $(\widetilde{\mu}_1,m_{{\sigma}^{(1)}_2}\widetilde{\mu}_1)$ if $n_1 \leq n_2-1$, and this system is Nikishin. Hence (ii) shows that for such indices $\widetilde{A}_{n_2,n_1}^{(2)}$ has all zeros in $\mathring{\Gamma}_2$, and they are all simple. This proves (i).
\end{proof}

\section{Zero interlacing}\label{ss:interlacing}

\subsection{Zero interlacing for Type II multiple orthogonal polynomials}\label{ss:interlacingII}
\hfill\\

Let $(x-z_0)^2\bm\mu$ stand for the system $((x-z_0)^2\mu_1,\dots,(x-z_0)^2\mu_r)$, where $(x-z_0)^2\mu_j$ is the double Christoffel transform of $\mu$, see Section~\ref{ss:mFunctions}. 

\begin{thm}\label{thm:interlacing type II} 
    Assume $\bm{n}$ and $\bm{n}+\bm{e}_j$ are normal for $\bm{\mu}$. Then 
    \begin{equation}
        \calZ[W(P_{\bm{n}+\bm{e}_j},P_{\bm{n}})] = \left\{
        z_0\in\bbC: \bm{n} \mbox{ is not normal for } (x-z_0)^2\bm\mu \right\}.
    \end{equation}
    In particular, if $P_{\bm{n}}$ 
    {is} real-rooted, then $P_{\bm{n}+\bm{e}_j} \sim P_{\bm{n}}$ if and only if $\bm{n}$ is normal with respect to $(x-z_0)^2\bm\mu$ for every $z_0 \in \bbR$.
\end{thm}
\begin{rem}\label{rem:type II interlacing}
    Similarly, if $\bm{n}+\bm{e}_j$, $\bm{n}+\bm{e}_k$ and $\bm{n}+\bm{e}_j+\bm{e}_k$ are normal for $\bm{\mu}$, $j \neq k$, then
    \begin{equation}
        \qquad \calZ[W(P_{\bm{n}+\bm{e}_j},P_{\bm{n}+\bm{e}_k})] = \left\{
        z_0\in\bbC: \bm{n} \mbox{ is not normal for } (x-z_0)^2\bm\mu \right\}.
    \end{equation}
    In particular, if $P_{\bm{n}+\bm{e}_j}$ 
    {is} real-rooted, then $P_{\bm{n}+\bm{e}_j} \sim P_{\bm{n}+\bm{e}_k}$ if and only if $\bm{n}$ is normal with respect to $(x-z_0)^2\bm\mu$ for every $z_0 \in \bbR$. 
\end{rem}
\begin{proof}
    First note that $W(P_{\bm{n}+\bm{e}_j},P_{\bm{n}};z_0) = 0$ if and only if the system of equations
    \begin{equation}
        \begin{cases}
            a P_{\bm{n}+\bm{e}_j}(z_0) + bP_{\bm{n}}(z_0) = 0 \\
            aP_{\bm{n}+\bm{e}_j}'(z_0) + bP_{\bm{n}}'(z_0) = 0
        \end{cases}
    \end{equation}
holds for some $(a,b) \neq (0,0)$. This is equivalent to $aP_{\bm{n}+\bm{e}_j}(x) + bP_{\bm{n}}(x) = (x-z_0)^2Q(x)$ for some polynomial $Q \neq 0$. Such a $Q$ would satisfy 
\begin{equation}\label{eq:orthogonality double christoffel}
    \int Q(x)x^p(x-z_0)^2d\mu_j(x) = 0, \qquad p = 0,\dots,n_j-1,
\end{equation}
but $\deg{Q} < \abs{\bm{n}}$, which contradicts to the normality of $\bm{n}$ for $(x-z_0)^2\bm{\mu}$. 

Conversely, if $\bm{n}$ is not normal for $(x-z_0)^2\bm{\mu}$, then there is some $Q \neq 0$ with $\deg{Q} < \abs{\bm{n}}$ satisfying \eqref{eq:orthogonality double christoffel}. Pick $a$ such that $(x-z_0)^2Q(x)-a P_{\bm{n}+\bm{e}_j}$ has at most degree $\abs{\bm{n}}$. By comparing orthogonality relations we find that $(x-z_0)^2Q(x)-a P_{\bm{n}+\bm{e}_j} = bP_{\bm{n}}$, assuming $\bm{n}$ is normal with respect to $\bm{\mu}$. Since $Q \neq 0$ we must have $(a,b) \neq (0,0)$, so $W(P_{\bm{n}+\bm{e}_j},P_{\bm{n}};z_0) = 0$. This concludes the proof, if we note that the interlacing follows from Lemma \ref{lem:wronskian interlacing}.

For Remark \ref{rem:type II interlacing}, the same proof works if we note that $P_{\bm{n}+\bm{e}_j}$ and $P_{\bm{n}+\bm{e}_k}$ are linearly independent when $\bm{n}+\bm{e}_j+\bm{e}_k$ is normal (the converse also holds). To see this, assume $P_{\bm{n}+\bm{e}_j}$ is a non-zero multiple of $P_{\bm{n}+\bm{e}_k}$ or vice versa. Then $P_{\bm{n}+\bm{e}_j}$ and $P_{\bm{n}+\bm{e}_k}$ satisfy all the orthogonality relations for the index $\bm{n}+\bm{e}_j+\bm{e}_k$, but $\deg{P_{\bm{n}+\bm{e}_j}} < \abs{\bm{n}+\bm{e}_j+\bm{e}_k}$, so $\bm{n}+\bm{e}_j+\bm{e}_k$ cannot be normal. 
Note that the linear independence of $P_{\bm{n}}$ and $P_{\bm{n}+\bm{e}_j}$ was immediate since they have different degrees.

\end{proof}

\begin{cor}
    Let $\bm{\mu}$ be an Angelesco system. 
    Then $P_{\bm{n}+\bm{e}_j} \sim P_{\bm{n}}$ and $P_{\bm{n}+\bm{e}_j} \sim P_{\bm{n}+\bm{e}_k}$ for any $\bm{n}$ and $j$, and any $k \neq j$. 
\end{cor}
\begin{proof}
     $(x-z_0)^2\bm{\mu}$ is Angelesco for any $z_0$, so Theorem~\ref{thm:interlacing type II} applies, along with Remark \ref{rem:type II interlacing}.
\end{proof}

\begin{cor}\label{cor:ATtypeInterlacingII}
    Let $\bm{\mu}$ be an AT system on $\Gamma$ for $\bm{n}$ and $\bm{n}+\bm{e}_j$. 
    Then zeros of $P_{\bm{n}}$ and $P_{\bm{n}+\bm{e}_j}$ interlace. If $\bm{\mu}$ is AT on $\Gamma$ for $\bm{n}+\bm{e}_j$, $\bm{n}+\bm{e}_k$ and $\bm{n}+\bm{e}_j+\bm{e}_k$, then $P_{\bm{n}+\bm{e}_j} \sim P_{\bm{n}+\bm{e}_k}$.
\end{cor}
\begin{proof}
     If $\bm{\mu}$ is AT for $\bm{k}$ then $(x-z_0)^2\bm{\mu}$ is 
     also AT for $\bm{k}$ (just replace the reference measure $\mu$, see Definition~\ref{def:AT}, with $(x-z_0)^2\mu$), so Theorem~\ref{thm:interlacing type II} applies.     
\end{proof}

\subsection{Zero interlacing for Type I multiple orthogonal polynomials}\label{ss:interlacingI}
\hfill\\

\begin{thm}\label{thm:interlacing type I}
    Assume that $\bm{n}$, $\bm{n}-\bm{e}_1$, and $\bm{n}-\bm{e}_
    \ell$ are normal for  $\bm{\mu}=(\mu_1,\mu_2,\dots,\mu_r)$. 
    Then
    \begin{multline}\label{eq:WronskiNew}
        \calZ[W(A_{\bm{n}}^{(1)},A_{\bm{n}-\bm{e}_\ell}^{(1)})] 
        \\
        =
        \left\{
        z_0 \in \bbC: \bm{n} - 2\bm{e}_1 \mbox{ is not normal for } ((x-z_0)^2\mu_1,\mu_2,\dots,\mu_r)
        \right\}.
    \end{multline}

    If $A_{\bm{n}}^{(1)}$ is real-rooted, then $A_{\bm{n}}^{(1)} \sim A_{\bm{n}-\bm{e}_\ell}^{(1)}$ if and only if $\bm{n}-2\bm{e}_1$ is normal for the system 
    $((x-z_0)^2\mu_1,\mu_2,\dots,\mu_r)$
    for every $z_0 \in \bbR$. 
    
    Analogous statements hold for $A_{\bm{n}}^{(j)}$ for $j = 2,\dots,r$, if we instead transform the $j$-th measure.
\end{thm}
\begin{rem}
    Assuming additionally the $\bm{\mu}$-normality of 
    $\bm{n}-\bm{e}_k$ and $\bm{n}-\bm{e}_k-\bm{e}_l$, then 
    $A_{\bm{n}-\bm{e}_\ell}^{(1)}\sim A_{\bm{n}-\bm{e}_k}^{(1)}$ if and only if $\bm{n}-2\bm{e}_1$ is normal for 
    $((x-z_0)^2\mu_1,\mu_2,\dots,\mu_r)$
    for every $z_0 \in \bbR$. 
\end{rem}
\begin{proof}
     Fix any $z_0\in\bbC$ and denote $\widehat{\widehat{\bm{\mu}}}$ to be $((x-z_0)^2\mu_1,\mu_2,\dots,\mu_r)$.     
     Note that $W(A_{\bm{n}}^{(1)},A_{\bm{n}-\bm{e}_\ell}^{(1)};z_0) = 0$ if and only if $aA_{\bm{n}}^{(1)}(x)+bA_{\bm{n}-\bm{e}_\ell}^{(1)}(x) = (x-z_0)^2B(x)$ for some $(a,b) \neq (0,0)$ and some polynomial $B$ with $\deg{B} \leq n_1-3$. Note that $\bm{A}_{\bm{n}}$ and $\bm{A}_{\bm{n}-\bm{e}_\ell}$ are linearly independent, since
    \begin{equation}\label{eq:indep}
        \sum_{j = 1}^r \int A_{\bm{n}}^{(j)}(x)x^{\abs{\bm{n}}-2}d\mu_j(x) = 0 \neq \sum_{j = 1}^r \int A_{\bm{n}-\bm{e}_\ell}^{(j)}(x)x^{\abs{\bm{n}}-2}d\mu_j(x).
    \end{equation}
    Then the vector $(B,aA_{\bm{n}}^{(2)}+bA_{\bm{n}-\bm{e}_\ell}^{(2)},\dots,aA_{\bm{n}}^{(r)}+bA_{\bm{n}-\bm{e}_\ell}^{(r)})$ is non-zero and satisfies all the degree and orthogonality conditions for the index $\bm{n}-2\bm{e}_1$ with respect to $\widehat{\widehat{\bm{\mu}}}$. Since we also have
    \begin{equation}
        \sum_{j = 1}^r \int (aA_{\bm{n}}^{(j)}(x)+bA_{\bm{n}-\bm{e}_\ell}^{(j)}(x))x^{\abs{\bm{n}}-3}d\mu_j(x) = 0,
    \end{equation}
    we see that $\bm{n}-2\bm{e}_1$ is not normal for $\widehat{\widehat{\bm{\mu}}}$ by Remark~\ref{rem:normality}. 

    Conversely, if $\bm{n}-2\bm{e}_1$ is not normal for $\widehat{\widehat{\bm{\mu}}}$, then there is some vector $\bm{B} = (B^{(1)},B^{(2)},\dots,B^{(r)}) \neq \bm{0}$ with $\deg{B^{(1)}} \leq n_1-3$ and $\deg{B^{(j)}} \leq n_j-1$ for $j = 2,\dots,r$, such that 
    \begin{equation}
        \sum_{j = 1}^r \int B^{(j)}(x)x^p d\widehat{\widehat{{\mu}}}_j(x) = 0, \qquad p = 0,\dots,\abs{\bm{n}}-3.
    \end{equation}
    Pick $a$ such that $((x-z_0)^2B^{(1)}(x),B^{(2)}(x)\dots,B^{(r)}(x)) - a\bm{A}_{\bm{n}-\bm{e}_\ell}(x)$ satisfies the same degree and orthogonality conditions as $\bm{A}_{\bm{n}}$ (with the correct choice of $a$ we get the one missing orthogonality condition). This shows that $aA_{\bm{n}}^{(1)}(x)+bA_{\bm{n}-\bm{e}_\ell}^{(1)}(x) = (x-z_0)^2B^{(1)}(x)$ for some $(a,b) \neq (0,0)$ and completes the proof of \eqref{eq:WronskiNew}. 

    The statement on interlacing then follows from Lemma \ref{lem:wronskian interlacing} since normality of $\bm{n}-\bm{e}_1$  gives us $\deg A^{(1)}_{\bm{n}} = n_1-1$ (use~\cite[Cor 23.1.1]{Ismail}), and $\deg A^{(1)}_{\bm{n}-\bm{e}_\ell}\le n_1-1\le \deg A^{(1)}_{\bm{n}}+1$.
\end{proof}

\begin{cor}
    Let $\bm{\mu}$ be an Angelesco system.
    Then $A^{(i)}_{\bm{n}} \sim A^{(i)}_{\bm{n}+\bm{e}_j}$ and $A^{(i)}_{\bm{n}+\bm{e}_j} \sim A^{(i)}_{\bm{n}+\bm{e}_k}$, assuming $n_i \geq 1$ and $j \neq k$.
\end{cor}
\begin{proof}
      $((x-z_0)^2\mu_1,\mu_2,\ldots,\mu_r)$  is Angelesco for any $z_0$, so Theorem~\ref{thm:interlacing type I} applies.
\end{proof}

For AT systems real-rootedness of type I polynomials does not have to hold, so one cannot expect interlacing, of course. For Nikishin systems, Theorem~\ref{thm:interlacing type I} can be used, however.

\begin{thm}\label{thm:NikishintypeInterlacingI}
    Let $(\mu_1,\mu_2)=\calN(\sigma_1,\sigma_2)$ be a Nikishin system. 
    Then
    \begin{enumerate}[(i)]
        \item 
        $A^{(1)}_{\bm{n}}$, $A^{(1)}_{\bm{n}-\bm{e}_1}$, and $A^{(1)}_{\bm{n}-\bm{e}_2}$ are pairwise  interlacing if $n_1 +1\leq n_2$. 
        \item $A^{(2)}_{\bm{n}}$, $A^{(2)}_{\bm{n}-\bm{e}_1}$, and $A^{(2)}_{\bm{n}-\bm{e}_2}$ are pairwise  interlacing if  $n_1+1\ge n_2$.
    \end{enumerate}
\end{thm}
\begin{rem}\label{rem:oneMoreDiagonal}
    In particular, this shows that the polynomials $A_{\bm{n}}^{(1)}$ with $n_1 = n_2$ and $A_{\bm{n}}^{(2)}$ with $n_1 +2 = n_2$  are real-rooted, which improves the statement from Theorem~\ref{thm:NikishintypeI}, see Remark~\ref{rem:simulationsNikishin}. Interlacing also shows that all except potentially one of the real roots of each of these polynomials belong to $\mathring{\Gamma}_2$.
\end{rem}
\begin{proof}
Similarly to the proof of Theorem \ref{thm:NikishintypeI} (ii), we get 
\begin{equation}
    (\mu_1,(x-z_0)^2\mu_2) = (\mu_1,m_{\widetilde{\sigma}}(x)\mu_1 + Q(x)\mu_1), 
\end{equation}
with $\supp\, \widetilde{\sigma} \subseteq \supp\, \sigma$ and $\deg{Q} \leq 1$. Lemma \ref{lem:perturbation}  then gives us normality of this system for the indices with $n_2\le n_1-1$. Index $\bm{n}-2\bm{e}_2$ is then normal if $n_2-2\le n_1-1$ which together with Theorem \ref{thm:interlacing type I} proves (ii).

To prove (i), one can follow the same arguments as in the proof of Theorem \ref{thm:NikishintypeI} (i).
\end{proof}

\section{Wronskians of higher order}\label{ss:Wronskians}

It should come as no surprise now that normality of  Christoffel transforms of higher degree are connected to the zeros of the Wronskians of higher order. Christoffel transforms of multiple orthogonal polynomials were the topic of~\cite{ADMVA,BFM22,KV1,KV2,ManRoj}.

Let us fix an ``increasing path'' of multi-indices, that is a sequence $\{\bm{n}_s\}_{s=1}^{\ell}$ with $l\in \{1,2,\ldots\}$, $\bm{n}_s\in\bbN^r$, so that 
$$
\bm{n}_{s+1} = \bm{n}_{s} + \bm{e}_{j_s}
$$
for every $s=1,1,\ldots,l-1$, and for some $j_s \in \{1,\ldots,r\}$.

\begin{thm}
    Suppose each multi-index in $\{\bm{n}_s\}_{s=1}^{\ell}$ is $\bm{\mu}$-normal. Then 
    \begin{equation}
        \calZ[W(P_{\bm{n}_1},\ldots,P_{\bm{n}_\ell};x)]
   =
        \left\{
        z_0\in\bbC: \bm{n}_1 \mbox{ is not normal for } (x-z_0)^\ell \bm\mu \right\}.
    \end{equation}
\end{thm}
We skip the proof since it wollows the same steps as the proof of Theorem~\ref{thm:interlacing type II}. This result immediately gives us the following corollary. 
\begin{cor}
    Let $\bm{\mu}$ be an Angelesco, AT, or Nikishin system, and $\ell$ be even. Then $W(P_{\bm{n}_1},\ldots,P_{\bm{n}_\ell};x)$ has no real zeros.
\end{cor}
For AT systems this was shown by Zhang and Filipuk in~\cite{ZhaFil}, and for $r=1$ this goes back to the classical result of Karlin and Szeg\H{o}~\cite[Thm 1]{KarSze}.

A similar result can be stated for type I polynomials as follows. 

\begin{thm}
    Suppose each multi-index in $\{\bm{n}_s\}_{s=1}^{\ell}$ is $\bm{\mu}$-normal, $(\bm{n}_1)_1\ge 2$. Then 
    \begin{multline}
        \calZ[W(A_{\bm{n}_1}^{(1)},\ldots,A_{\bm{n}_\ell}^{(1)};x)] 
        \\
        =
        \left\{
        z_0 \in \bbC: \bm{n}_1 - \bm{e}_1 \mbox{ is not normal for } ((x-z_0)^\ell\mu_1,\mu_2,\dots,\mu_r)
        \right\}.
    \end{multline}
    Analogous statements hold for $A_{\bm{k}}^{(j)}$ for $j = 2,\dots,r$, if we instead transform the $j$-th measure.
\end{thm}
\begin{cor}
    Let $\bm{\mu}$ be an Angelesco system, and $\ell$ be even. Then, for any $j=1,\ldots,r$, the Wronskian $W(A_{\bm{n}_1}^{(j)},\ldots,A_{\bm{n}_\ell}^{(j)};x)$ has no real zeros.
\end{cor}
One can show a similar statement for Nikishin systems for indices in certain cones, similarly as in Theorem~\ref{thm:NikishintypeInterlacingI}.



\bibsection


\end{document}